\documentclass[12pt]{article}

\usepackage{bm}
\usepackage{amsxtra}
\usepackage{latexsym}
\usepackage{amsthm}
\usepackage{amssymb}
\usepackage{amsmath}
\usepackage{graphicx}
\usepackage{epsfig}
\usepackage{afterpage}
\newtheorem{thm}{Theorem}[]

\newtheorem{rem}[]{Remark}
 \newcommand{\thmref}[1]{Theorem~\ref{#1}}

\newcommand{\R}{{\mathbb R}}

\newcommand{\bee}{\begin{equation*}}
\newcommand{\eee}{\end{equation*}}
\newcommand{\be}{\begin{equation}}
\newcommand{\ee}{\end{equation}}

\def\Re{\mathop{\rm Re}}

\date{}

\index{Ramm, A. G.}

\begin{document}


\centerline{}

\centerline{}

\centerline {\Large{\bf Large-time behavior of solutions to evolution
problems}}

\centerline{}

\centerline{\bf {A. G. Ramm$\dag$}}

\centerline{}

\centerline{$\dag$Mathematics Department, Kansas State University,}

\centerline{Manhattan, KS 66506-2602, USA}

\centerline{(E-mail: {\tt ramm@math.ksu.edu})}

\renewcommand{\thefootnote}{\fnsymbol{footnote}}
\footnotetext[3]{Corresponding author. Email: ramm@math.ksu.edu}

\begin{abstract} Large time behavior of solutions to abstract
differential equations is studied. The results give sufficient condition
for the solution to an abstract dynamical system (evolution problem)
not to exibit chaotic behavior.
The corresponding evolution
problem is:
$$\dot{u}=A(t)u+F(t,u)+b(t), \quad t\ge 0; \quad u(0)=u_0. \qquad (*)$$
Here $\dot{u}:=\frac {du}{dt}$, $u=u(t)\in H$, $H$ is a Hilbert space,
$t\in \R_+:=[0,\infty)$,
$A(t)$ is a linear dissipative operator: Re$(A(t)u,u)\le -\gamma(t)(u,u)$,
$F(t,u)$ is a nonlinear operator,
$\|F(t,u)\|\le c_0\|u\|^p$, $p>1$, $c_0,p$ are positive constants,
$\|b(t)\|\le
\beta(t),$ $\beta(t)\ge 0$
is a continuous function.

Sufficient conditions are given for the solution $u(t)$
to problem (*) to exist for all $t\ge0$, to be bounded uniformly on
$\R_+$, and a bound on $\|u(t)\|$ is given. This bound implies
the relation $\lim_{t\to \infty}\|u(t)\|=0$ under suitable
conditions on $\gamma(t)$ and $\beta(t)$.

The basic technical tool in this work is the following
nonlinear inequality:
$$ \dot{g}(t)\leq
-\gamma(t)g(t)+\alpha(t,g(t))+\beta(t),\ t\geq 0;\quad g(0)=g_0,
$$
which holds on any interval $[0,T)$ on which $g(t)\ge 0$ exists and has
bounded derivative
from the right, $\dot{g}(t):=\lim_{s\to +0}\frac{g(t+s)-g(t)}{s}$.
It is assumed that  $\gamma(t)$,
and $\beta(t)$ are real-valued, continuous
functions of $t$, defined on
$\R_+:=[0,\infty)$,  the function $\alpha(t,g)$ is
defined for all $t\in \R_+$,  locally Lipschitz
with
respect to $g$ uniformly with respect to $t$ on any compact subsets
$[0,T]$, $T<\infty$.

If there exists a function $\mu(t)>0$,  $\mu(t)\in C^1(\R_+)$, such that
$$\alpha\left(t,\frac{1}{\mu(t)}\right)+\beta(t)\leq
\frac{1}{\mu(t)}\left(\gamma(t)-\frac{\dot{\mu}(t)}{\mu(t)}\right),\quad
\forall t\ge 0;\quad \mu(0)g(0)\leq 1,$$ then $g(t)$ exists on all of
$\R_+$, that is $T=\infty$,  and the following estimate holds:
$$0\leq g(t)\le \frac 1{\mu(t)},\quad \forall t\geq 0. $$
If $\mu(0)g(0)< 1$, then $0\leq g(t)< \frac 1{\mu(t)},\quad \forall
t\geq 0. $

\end{abstract}


\newpage
\section{Introduction}
A classical area of study is stability of solutions to evolution
problems. We identify an evolution problem with an abstract
dynamical system. An evolution problem is described by an equation
\be\label{ea} \dot{u}(t)=F_1(t,u),\qquad u(0)=u_0. \ee Here $F_1: X\to
X$ is a nonlinear operator in a Banach space $X$, $\dot{u}=\dot{u}(t)=\frac{du}{dt}$. Quite often it is
convenient to assume  $X$ to be a Hilbert space $H$, because the
energy is often interpreted as a quantity $(u,u)$ in a suitable
Hilbert space. Suppose that $F_1(t,0)=0$ and $u_0=0$. Then $u=0$ is
a solution to \eqref{ea}. A. M. Lyapunov in 1892 published a classical
work on stability of motion, where he studied equation \eqref{ea} in
the case $X=\R^n$ and $F_1$ analytic function of $u$. If
$F_1(t,0)=0$, and $F_1$ is twice Fr\'echet differentiable, then one
can write $F_1(t,u)=A(t)u+F(t,u)$, where $A(t)$ is a linear operator
in $X$ and $\|F(t,u)\|=O(\|u\|^2),\,\,\, \|u\|\to 0.$ This
representation is a linearization of $F_1$ around the point $u=0$.
Lyapunov defined the notion of stability (Lyapunov stability) of the
equilibrium solution $u=0$ towards small perturbation of the data
$u_0$. He calls this solution stable (Lyapunov stable), if for any
$\epsilon>0$ there is a $\delta=\delta(\epsilon)>0$ such that
if inequality $\|u_0\|< \delta$ holds then $\sup_{t\ge 0}\|u(t)\|<
\epsilon$. Note that this definition implies the global existence of
the solution to problem \eqref{ea} for all $u_0$ in the ball
$\|u_0\|< \delta$.

 The equilibrium solution $u=0$ is unstable if
it is not Lyapunov stable. This means that there is an $\epsilon>0$
such that for any $\delta>0$ there is a $u_0$, $\|u_0\|< \delta$,
and a $t_\delta>0$ such that $\|u(t_\delta)\|\ge \epsilon$.

One can give similar definitions for stability and instability of a
solution to problem \eqref{ea} with $u_0\neq 0$. In this case one
calls the solution $u=u(t;u_0)$ stable if all the solutions $u(t;
w_0)$ to problem \eqref{ea}, with $w_0$ in place of $u_0$, exist for
all $t\ge 0$ and satisfy the inequality $\sup_{t\ge
0}\|u(t;u_0)-u(t; w_0)\|< \epsilon$ provided that $\|u_0-w_0\|<
\delta$.

A solution $u(t;u_0)$ is called asymptotically stable if it is
stable and there is a $\delta>0$ such that all the solutions $u(t;
w_0)$ with $\|u_0-w_0\|< \delta$ satisfy the relation $\lim_{t\to
\infty}\|u(t;u_0)-u(t; w_0)\|=0$.

The equilibrium solution $u=0$ is asymptotically stable if it is
stable and there is a $\delta>0$ such that all the solutions $u(t;
u_0)$ with $\|u_0\|< \delta$ satisfy the relation $\lim_{t\to
\infty}\|u(t;u_0)\|=0$.

Consider problem \eqref{ea} with $F_1(t,u)+\phi(t,u)$ in place of
$F_1(t,u)$. The term $\phi(t,u)$ is called persistently acting
perturbations. The equilibrium solution $u=0$ is called stable with
respect to persistently acting perturbations if for any $\epsilon>0$
there exists a $\delta=\delta(\epsilon)>0$ such that if
$\|\phi(t,u)\|< \delta $ and $\|u_0\|<\delta$, then $\sup_{t\ge
0}\|u(t;u_0)\|< \epsilon$.

Stability of the solutions and their behavior as $t\to \infty$ are
of interest in a study of dynamical systems. For example, if the
equilibrium solution is asymptotically stable, then it does not have
chaotic behavior.

If $A(t)=A$ is independent of time and $X=\R^n$, then Lyapunov
obtained classical results on the stability of the equilibrium
solution to problem \eqref{ea}. He assumed that $F$ is analytic
with respect to $u\in \R^n$, so that $|F(t,u)|\le c|u|^2$ in a
neighborhood of the origin, and $c>0$ is a constant. Lyapunov has
proved that {\it if the spectrum
$\sigma(A)$ of $A$ lies in the half-plane Re$z<0$, then the
equilibrium solution $u=0$ is asymptotically stable, and if at least
one eigenvalue of $A$ lies in the half-plane Re$z>0$, then the
equilibrium solution is unstable.}

If  some of the eigenvalues of $A$ lie on the imaginary axis and
$F=0$, so that problem \eqref{ea} is linear, and if the
corresponding Jordan cells consist of just one element, then the
equilibrium solution is stable. Otherwise it is unstable.

 Thus,  a necessary and sufficient condition for Lyapunov stability
of the equilibrium solution of the linear equation $\dot{u}=Au$ in
$\R^n$ is known: the spectrum of $A$ has to lie in the left complex
half-plane:  $\sigma\subset \{Z: Re z\le 0\}$,  and the Jordan cells
corresponding to purely imaginary eigenvalues of $A$ have to consist
of just one element.

 If $F\not\equiv 0$, then, in general, when the
spectrum of $A$ lies in the left half plane of the complex plane,
and some eigenvalues of $A$ lie on the imaginary axis,  the
stability cannot be decided by the linearized part $A$ of $F_1$
only. One can give examples of $A$ such that the nonlinear part $F$
can be chosen so that the equilibrium solution $u=0$ is stable, and
$F$ can also be chosen so that this solution is unstable. For
instance, consider $\dot{u}=cu^3$, where $c=const$. This equation
can be solved analytically by separation of variables. The result is
$u(t)=[u^{-2}(0)-2ct]^{-0.5}$. Therefore, if $c<0$ and $|u(0)|\leq
\delta$, $\delta>0$, then the solution exists for all $t\ge 0$, and
is asymptotically stable. But if $c>0$, then the solution blows up
at a finite time $t_b$, the blow-up time, and $t_b=[2cu^2(0)]^{-1}$.
In this case  the zero solution is unstable.

 If $A=A(t)$ the stability theory is more complicated. The
case of periodic $A(t)$ was studied much due to its importance
in many applications (see \cite{YS}).

The stability theory in infinite-dimensional spaces, for example, in
Hilbert and Banach spaces, was developed in the second half of the
20-th century, see \cite{DK} and references therein. Again, the
location of the spectrum of $A(t)$ plays an important role in this
theory.

The basic {\it novel points} of the theory presented below include
sufficient conditions for the stability and asymptotic stability of
the equilibrium solution to abstract evolution problem  \eqref{ea}
in a Hilbert space when $\sigma (A(t))$ may lie in the right
half-plane for some or all  $t>0$, but $\sup \sigma (Re A(t))\to 0$
as $t\to \infty$. Therefore, our results are new even in the finite-dimensional spaces.

The technical tool, on which our study is based, is
a new nonlinear differential inequality. The results are stated in
several theorems and illustrated by several examples. These results
are taken from the cited papers by the author (see [14]-[22]), and,
especially, from paper \cite{R605}.  In the joint papers
by the author's student  N. S.Hoang and the author one can find
various additional results on nonlinear inequalities (see [5]-[10]).
Some versions of this inequality has been used in the monographs
[12] and [13], where the Dynamical Systems Method (DSM) for solving
operator equations was developed.

The literature on stability of solutions to evolution problems and
their behavior at large times is enormous, and we refer the reader
only to the papers and books directly related to the novel points
mentioned above.

Consider an abstract nonlinear evolution problem \be\label{e1}
\dot{u}=A(t)u+F(t,u)+b(t),  \qquad \dot{u}:=\frac {du}{dt}, \ee
\be\label{e2} u(0)=u_0,\ee where $u(t)$ is a function with values in
a Hilbert space $H$, $A(t)$ is a linear bounded dissipative operator
in $H$, which satisfies inequality \be\label{e3}
\text{Re}(A(t)u,u)\leq -\gamma(t)\|u\|^2,\qquad t\geq 0;\qquad
\forall u\in H, \ee where $F(t,u)$ is a nonlinear map in $H$,
\be\label{e4} \|F(t,u)\|\leq c_0\|u(t)\|^p,\qquad p>1, \ee
\be\label{e5} \|b(t)\|\leq \beta(t), \ee $\gamma(t)>0$ and
$\beta(t)\ge 0$ are continuous function, defined on all of
$\R_+:=[0,\infty)$, $c_0>0$ and $p>1$ are constants.

Recall that a linear operator $A$ in a Hilbert space is called dissipative
if Re$(Au,u)\le 0$ for all $u\in D(A)$, where $D(A)$ is the domain of
definition of $A$. Dissipative operators are important because they
describe systems in which energy is dissipating, for example, due to
friction or other physical reasons. Passive nonlinear networks
can be described by equation \eqref{e1} with a dissipative
linear operator $A(t)$, see \cite{R129}, \cite{R118}, Chapter 3, and
\cite{T}.

Let $\sigma:=\sigma(A(t))$ denote the spectrum of the linear operator
$A(t)$,  $\Pi:=\{z: Re z<0\}$, $\ell:=\{z: Re z=0\}$, and $\rho(\sigma,
\ell)$  denote the distance between sets $\sigma$ and $\ell$.
We assume that
\be\label{e6}
\sigma\subset \Pi,
\ee
but we allow $\lim_{t\to \infty}
\rho(\sigma,\ell)=0$. This is the basic {\it novel} point in our theory.
The usual assumption in stability theory (see, e.g., \cite{DK}) is
$ \sup_{z\in \sigma} Re z\le -\gamma_0$, where $\gamma_0=const>0$.
For example, if $A(t)=A^*(t)$, where $A^*$ is the adjoint operator,
and if the spectrum of $A(t)$ consists of eigenvalues $\lambda_j(t)$,
$0\ge \lambda_j(t)\ge \lambda_{j+1}(t)$, then,
we allow $\lim_{t\to \infty} \lambda_1(t)=0$.
This is in contrast with the usual
theory, where the assumption is $\lambda_1(t)\le -\gamma_0$,
$\gamma_0>0$ is a constant, is used.

Moreover, our results cover the case, apparently not considered
earlier in the literature, when Re$(A(t)u,u)\le \gamma(t)$ with
$\gamma(t)>0$, $\lim_{t\to \infty}\gamma(t)=0$. This means that the
spectrum of $A(t)$ may be located in the half-plane Re$z\le \gamma (t)$,
where $\gamma (t)>0$, but  $\lim_{t\to \infty}\gamma(t)=0$.

Our goal is to give
sufficient conditions for the existence and uniqueness of the solution
to problem \eqref{e1}-\eqref{e2} for all $t\ge0$, that is, for global
existence of $u(t)$, for boundedness of $\sup_{t\ge0}\|u(t)\|<\infty,$
or to the relation $\lim_{t\to \infty}\|u(t)\|=0$.

If $b(t)=0$ in \eqref{e1}, then $u(t)=0$ solves equation \eqref{e1}
and $u(0)=0$. This equation is called zero solution to \eqref{e1}
with $b(t)=0$.

Recall that the zero solution to equation \eqref{e1} with $b(t)=0$
is called Lyapunov stable if for any $\epsilon>0$, however small,
one can find a $\delta=\delta(\epsilon)>0$, such that if $\|u_0\|\le
\delta$, then the solution to Cauchy problem \eqref{e1}-\eqref{e2}
satisfies the estimate $\sup_{t\ge 0}\|u(t)\|\le \epsilon$. If, in
addition, $\lim_{t\to \infty}\|u(t)\|=0$, then the zero solution to
equation \eqref{e6} is called asymptotically stable.

If $b(t)\not\equiv 0$, then one says that  \eqref{e1}-\eqref{e2} is the
problem with persistently acting perturbations.
The zero solution is called Lyapunov stable for problem
\eqref{e1}-\eqref{e2}
with persistently acting perturbations if for any $\epsilon>0$, however
small, one can find a $\delta=\delta(\epsilon)>0$, such that if
$\|u_0\|\le \delta$, and $\sup_{t\ge 0}\|b(t)\|\le\delta$, then the
solution to Cauchy problem
\eqref{e1}-\eqref{e2} satisfies the estimate $\sup_{t\ge 0}\|u(t)\|\le
\epsilon$.

We do not discuss here the method of Lyapunov functions for a study
of stability (see, for example,  \cite{Kr}, \cite{Ro}).

The approach, developed in this work, consists of reducing the
stability problems to some nonlinear differential inequality
and estimating the solutions to this inequality.

In Section 2 the formulation and a proof of two theorems, containing the
result concerning this inequality and its discrete analog, are given.
In Section 3 some results concerning Lyapunov stability
of zero  solution to equation  \eqref{e1} are obtained.
In Section 4 we derive stability results in the case when $\gamma(t)>0$.
This means that the linear operator $A(t)$ in \eqref{e1} may have spectrum
in the half-plane Re$z>0$.

Our results are
closely related to the Dynamical Systems Method (DSM), see
\cite{R485}, \cite{R577}, \cite{R585}, \cite{R606}. Recently
these results were applied to biological problems (\cite{R619}) and
to evolution equations with delay (\cite{R617}).

In the theory of chaos one of the reasons for the chaotic behavior
of a solution to an evolution problem to
appear is the lack of stability of solutions to this problem
(\cite{Da}, \cite{De}). The results presented in Section 3 can be
considered as sufficient conditions for chaotic behavior not to appear
in the evolution system described by problem \eqref{e1}-\eqref{e2}.

\section{A differential inequality}

In this Section an essentially self-contained proof is given of an
estimate for non-negative solutions of a nonlinear inequality
\be\label{e7}
\dot{g}(t)\leq -\gamma(t)g(t)+\alpha(t,g(t))+\beta(t),\ t\geq 0;\
g(0)=g_0;\quad \dot{g}:=\frac{dg}{dt}.  \ee
In Section 3 some of the many possible applications of this estimate
(see estimate  \eqref{e11} below)  are demonstrated.

It is  not assumed a priori that solutions $g(t)\ge 0$ to inequality
\eqref{e7} are defined on all of $\R_+$, that is, that  these solutions
exist globally. In Theorem 1 we give sufficient conditions for the
global existence of $g(t)$.
Moreover, under these conditions  a bound on $g(t)$ is given, see
estimate \eqref{e11} in Theorem 1.
This bound yields the relation $\lim_{t\to \infty}g(t)=0$
if $\lim_{t\to \infty}\mu(t)=\infty$ in \eqref{e11}.

Let us formulate our assumptions. We assume that $g(t)\ge 0$.
We {\it do not assume that the
functions $\gamma, \alpha$ and $\beta$ are non-negative}. However, in
many applications the functions $\alpha$ and $\beta$ are bounds on some
norms, and then these functions are non-negative. The function $\gamma(t)$
is often (but not always) non-negative. For example, this happens if  $\gamma(t)$ comes from an estimate of
the type $(Au,u)\ge \gamma (u,u)$. If the functions $\alpha$ and $\beta$
are bounds from above on some norms, then one may assume without loss of generality
that these functions are smooth, because one can approximate
a non-smooth function with an arbitrary accuracy by an infinitely smooth
function, and choose this smooth function to be greater than the function it approximates.

\noindent ${\bf Assumption\,\, A_1).}$ We assume that the function
$g(t)\geq 0$ is defined on some interval $[0,T)$, has a bounded
derivative $\dot{g}(t):=\lim_{s\to +0}\frac{g(t+s)-g(t)}{s}$ from
the right at any point of this interval, and  $g(t)$ satisfies
inequality \eqref{e7} at all $t$ at which $g(t)$ is defined. The
functions $\gamma(t)$, and $\beta(t)$,  are real-valued, defined on
all of $\R_+$ and continuous there.  The function $\alpha(t,g)$ is
continuous on $\R_+\times \R_+$ 
and locally Lipschitz with respect to $g$. This
means that
\be\label{e8}
|\alpha(t,g)-\alpha(t,h)|\leq L(T,M)|g-h|,
\ee
if $t\in[0,T]$, $|g|\leq M$ and $|h|\leq  M$. Here $M=const>0$
and $L(T,M)>0$ is a constant independent of $g$, $h$, and $t$.

\noindent ${\bf Assumption\,\, A_2).}$ There exists a $C^1(\R_+)$
function $\mu(t)>0$,  such that
\be\label{e9}
\alpha\left(t,\frac{1}{\mu(t)}\right)+\beta(t)\leq
\frac{1}{\mu(t)}\left(\gamma(t)-\frac{\dot{\mu}(t)}{\mu(t)}\right),\quad
\forall t\ge 0, \ee
and
\be\label{e10} \mu(0)g(0)\le 1. \ee

One can replace the initial point $t=0$ by some point $t_0\in \R$,
and assume that the interval of time is $[t_0, t_0+T)$, and that
inequalities hold for $t\ge t_0$, rather than for $t\ge 0$. The proofs and
the conclusions remain unchanged.

\begin{thm}\label{thm1}
If ${\it Assumptions\,  A_1)}$ and ${\it A_2)}$ hold, then any solution
$g(t)\ge 0$
to inequality \eqref{e7} exists on all of $\R_+$, i.e., $T=\infty$,
and satisfies the following estimate:
\be\label{e11}0\leq g(t)\le \frac{1}{\mu(t)}\quad \forall t\in \R_+. \ee
If $ \mu(0)g(0)< 1$, then $0\leq g(t)< \frac{1}{\mu(t)}\quad \forall
t\in \R_+.$
\end{thm}
\begin{rem}\label{rem1}
If $\lim_{t\to \infty} \mu(t)=\infty$, then $\lim_{t\to
\infty}g(t)=0$.
\end{rem}

\noindent {\it Proof of \thmref{thm1}.} Let us rewrite inequality
for $\mu$ as follows: \be\label{e12} -\gamma (t)\mu^{-1}(t)+\alpha(t,
\mu^{-1}(t))+\beta(t)\le \frac {d\mu^{-1}(t)}{dt}. \ee Let $\phi(t)$
solve the following Cauchy problem: \be\label{e13}
\dot{\phi}(t)=-\gamma (t)\phi(t)+\alpha(t, \phi(t))+\beta(t),\quad
t\ge 0; \quad \phi(0)=\phi_0. \ee
The assumption that $\alpha(t,g)$
is locally Lipschitz with respect to $g$ guarantees local existence and uniqueness of
the solution $\phi (t)$ to problem \eqref{e13}. From the comparison
result (see {\bf A Comparison Lemma} proved below) it follows
that \be\label{e14} \phi(t)\le
\mu^{-1}(t) \qquad \forall t\ge 0, \ee provided that $\phi(0)\le
\mu^{-1}(0)$, where $\phi(t)$ is the unique solution to problem
\eqref{e14}. Let us take $\phi(0)=g(0)$. Then  $\phi(0)\le
\mu^{-1}(0)$ by the assumption, and an inequality, similar to
\eqref{e14}, implies that \be\label{e15} g(t)\le \phi(t) \qquad t\in
[0,T). \ee Inequalities $\phi(0)\le \mu^{-1}(0)$, \eqref{e14}, and
\eqref{e15} imply \be\label{e16} g(t)\le \phi(t)\le \mu^{-1}(t),
\qquad t\in [0,T). \ee By the assumption, the function $\mu(t)$ is
defined for all $t\ge 0$ and is bounded on any compact subinterval
of the set $[0,\infty)$. Consequently, the functions $\phi(t)$ and
$g(t)\ge 0$ are defined for all $t\ge 0$, and estimate \eqref{e11}
is established.

If $g(0)< \mu^{-1}(0)$, then one obtains by a similar argument
the strict inequality $g(t)< \mu^{-1}(t), \quad t\ge 0$.

\thmref{thm1} is proved.\hfill $\Box$

Let us now prove the comparison result that was used above, see,
 for example, \cite{H}, Theorem III.4.1.
 \newpage

{\bf A Comparison Lemma.} {\it Let
$$\dot{\phi}(t)= f(t,\phi),\,
\,\, \phi(0)=\phi_0,\,\,\, (*)$$ and $$\dot{\psi}(t)= g(t,\psi), \,\,\,
\psi(0)=\psi_0. \,\,\, (**)$$
Assume $\psi_0\ge \phi_0$, and
$$g(t,x)\ge f(t,x)\qquad (***)$$ for
any $t$ and $x$ for which both $f$ and $g$ are defined. Assume that
$f$ and $g$ are continuous functions in a set $[0,s) \times (a,b)$,
 $\phi_0\in (a,b)$,   $\psi$ is the maximal solution to (**) and
$\phi$ is any solution to (*). Then $\phi(t)\le \psi(t)$ on the
maximal interval $[0,T)$ of the existence of both $\phi$ and
$\psi$.}

{\it Proof of the Comparison Lemma.} First, let us assume for
simplicity that problems (*) and (**) have a unique solution. Later
we will discard this simplifying assumption. If $f$ and $g$ satisfy
a local Lipschitz condition with respect to $\phi$, respectively,
$\psi$, then our simplifying assumption holds. Assume secondly, also
for simplicity, that $g(t,x)>f(t,x)$. Under this simplifying
assumption it is easy to prove the conclusion of the Lemma, because
the graph of $\psi$ must lie above the graph of $\psi$ for $t>0$.
Indeed, in a small neighborhood $[0,\delta)$, where $\delta>0$ is
sufficiently small, the graph of $\psi$ lies above the graph of
$\phi$. This is obviously true if $\phi_0<\psi_0$, because of the
continuity of $\phi$ and $\psi$. If $\phi_0=\psi_0$, then the graph
of $\psi$ lies above the graph of $\phi$ because
$\dot{\phi}(0)<\dot{\psi}(0)$ due to the assumption
$f(0,\phi_0)<g(0,\phi_0)=g(0,\psi_0)$. To check the last claim
assume that there is a point $t_1\in [0,T)$ such that
$\phi(t_1)=\psi(t_1)$, and $\phi(t)<\psi(t)$ for $t\in (0,t_1)$.
Then $\phi(t)-\phi(t_1)<\psi(t)-\psi(t_1)$. Divide this inequality
by $t-t_1<0$ and get
$$\frac{\phi(t)-\phi(t_1)}{t-t_1}>\frac{\psi(t)-\psi(t_1)}{t-t_1}.$$
Pass to the limit $t\to t_1$, $t<t_1$, in the above inequality, use
the differential equations for $\phi$ and $\psi$  and the equality
$\phi(t_1)=\psi(t_1)$, and obtain the following relation:
$$f(t_1,\phi(t_1))=\dot{\phi}(t_1)\ge
\dot{\psi}(t_1)=g(t_1, \psi(t_1))=g(t_1, \phi(t_1)).$$
This relation contradicts the assumption $f(t,x)<g(t,x)$. The
contradiction
proves the conclusion of the Comparison Lemma under the additional
assumption
$f(t,x)<g(t,x)$.

To prove the Comparison Lemma under the original
assumption $f(t,x)\le g(t,x)$,  let us consider problem (*) with $f$
replaced by $f_n:=f-\frac 1 n<f$. Let $\phi_n$  solve  problem (*)
with $f$ replaced by $f_n$, and with the same initial condition as
in (*). Since $f_n(t,x)<g(t,x)$, then, by what we have just proved,
it follows that $\phi_n(t)\le \psi(t)$ on the common interval
$[0,T_n)$ of the existence of $\phi_n$ and $\psi$. By the standard
result about continuous dependence of the solution to (*) on a
parameter, one concludes that $\lim_{n\to \infty} T_n=T$ and
$\lim_{n\to \infty}\phi_n(t)=\phi(t)$ for any $t\in [0,T)$.
Therefore, passing to the limit $n\to \infty$ in the inequality
$\phi_n(t)\le \psi(t)$ one gets the conclusion of the Comparison Lemma
under
the original assumption $f(t,x)\le g(t,x)$.

If the simplifying
assumption concerning uniqueness of the solutions to (*) and (**) is
dropped, then (*) and (**) may have many solutions. The limit of the
solution $\phi_n$ is the minimal solution to (*). If one considers
problem (**) with $g$ replaced by $g_n:=g+\frac 1 n>g$, and denotes
by $\psi_n$ the corresponding solution, then the limit $\lim_{n\to
\infty}\psi_n(t)=\psi(t)$ is the maximal solution to (**). In this
case the above argument yields the conclusion of the Lemma with
$\psi(t)$ being the maximal solution to (**), and $\phi(t)$ being
any solution to (*). The Comparison Lemma is proved. \hfill $\Box$

{\bf Remark 2.} {\it  If $\phi(t)$ is bounded from below for all
$t\ge 0$, so that $c\le \phi(t)$ for all $t\ge 0$, and $\psi(t)$
exists globally, that is, for all $t\ge 0$, then the inequality
$c\le \phi(t)\le \psi(t)$ and the continuity of $f(t,x)$ on the set
$[0,\infty)\times \R$ imply that any solution $\phi$ to (*) exists
globally. Indeed, if it would exist only on a finite interval
$[0,T)$ then it has to tend to infinity as $t\to T$, but this is
impossible because  the bound $c\le \phi(t)\le \psi(t)$ and the
global existence and continuity of $\psi$ do not allow $\phi(t)$ to
grow to infinity as $t\to T$.}

 Let us formulate and prove a {\it discrete version} of \thmref{thm1}.

\begin{thm}\label{thm2}
Assume that $g_n\geq 0$, $\alpha(n,g_n)\geq 0,$
\be\label{e21}
g_{n+1}\leq (1-h_n\gamma_n)g_n+h_n\alpha(n,g_n)+h_n\beta_n;\quad
h_n>0,\quad 0<h_n\gamma_n<1,\ee and $\alpha(n,g_n)\geq \alpha(n,p_n)$ if
$g_n\geq p_n$. If there exists a sequence $\mu_n> 0$ such that
\be\label{e22} \alpha(n,\frac{1}{\mu_n}) +\beta_n\leq
\frac{1}{\mu_n}(\gamma_n-\frac{\mu_{n+1}-\mu_n}{h_n\mu_n}),
\ee
and
\be\label{e23} g_0\leq \frac{1}{\mu_0}, \ee
then
\be\label{e24}
0\leq g_n\leq \frac{1}{\mu_n}, \qquad \forall n\geq 0. \ee
\end{thm}
\begin{proof}
For $n=0$ inequality \eqref{e24} holds because of \eqref{e23}.
Assume that it holds for all $n\leq m$ and let us check that then it holds
for $n=m+1$. If this is done, \thmref{thm2} is proved.\hfill $\Box$

Using the
inductive assumption, one gets:
\bee g_{m+1}\leq
(1-h_m\gamma_m)\frac{1}{\mu_m}+h_m\alpha(m,\frac{1}{\mu_m})+h_m\beta_m.
\eee
This and inequality \eqref{e22} imply:
\bee\begin{split}
g_{m+1}&\leq
(1-h_m\gamma_m)\frac{1}{\mu_m}+h_m\frac{1}{\mu_m}(\gamma_m-\frac{\mu_{m+1}-
\mu_m}{h_m\mu_m})\\
&=\mu_m^{-1}-\frac{\mu_{m+1}-\mu_m}{\mu_m^2}\le \mu_{m+1}^{-1}.
\end{split}\eee
The last inequality is obvious since it can be written as
$$-(\mu_m-\mu_{m+1})^2\le 0.$$
\thmref{thm2} is proved.
\end{proof}
\thmref{thm2} was formulated in \cite{R593} and proved in \cite{R558}.
We included for completeness a proof, which  is shorter than the one in
\cite{R558}.

Let us give a few simple examples of applications of Theorem 1.

{\bf Example 1.} Consider the inequality \be\label{e1x}\dot{g}(t)\le
tg-(t+1)^2g^2 -2(t+1)^{-2}. \ee Assume $g\ge 0$. Choose
$\mu(t)=t+1$. Then inequality \eqref{e9} holds if
$$(t+1)[-(t+1)^2(t+1)^{-2}-2(t+1)^{-2}]\le -t-(t+1)^{-1},$$ and
$g(0)\le 1$. Thus, inequality \eqref{e9} holds if
$$-t-1 -2(t+1)^{-1}\le -t -(t+1)^{-1}.$$ This inequality holds obviously.
Therefore, any $g\ge 0$, that satifies inequalities \eqref{e1x}  and
$g(0)\le1$, exists for all $t\ge 0$ and satisfies the estimate
$$0\le g(t)\le \frac {1} {t+1}.$$

In this example the
linearized problem $$\dot{g}(t)= tg -2(t+1)^{-2},\qquad g(0)=g_0,$$
has a unique solution $$g(t)=e^{t^2/2}[g(0)-2\int_0^t
e^{-\frac{s^2}{2}}(s+1)^{-2}ds].$$ This
solution tends to infinity as $t\to \infty$.

{\bf Example 2.} Consider a classical problem \be\label{e2x}
\dot{u}(t)=A(t)u+F(t,u), \qquad u(0)=u_0, \ee where $A(t)$ is a
linear operator in $\R^n$ and $F$ is a nonlinear operator. Assume
that $\Re(A(t)u,u)\le -\gamma (u,u)$, where $\gamma=const>0$, and
$||F(t,u)||\le c||u||^p$, $p=const>1$, $c=const>0$, and $||\cdot||$
is the norm of a vector in $\R^n$. We also assume that equation
\eqref{e2x} has the following property:

{\bf Property P}: {\it If a solution to \eqref{e2x} is defined  on
the maximal interval of its existence $[0,T)$ and $T<\infty$, then
$\lim_{t\to T-0}||u(t)||=\infty$. }

It is known (see, for example, \cite{H}), that  Property P
holds if $F(t,u)$ is a continuous function on $[0,T]\times \R^n$.

 By Peano's theorem the Cauchy problem
\be\label{e2y} \dot{u}(t)=f(t,u),\qquad u(0)=u_0, \ee where $u\in
\R^n$, has a local solution on an interval $[0,a)$, provided that
$f$ is a continuous function on $[0,T]\times D(u_0)$, where $a\in
(0,T)$ and $D(u_0)$ is a neighborhood of $u_0$.  This solution is
non-unique, in general.   One can give an explicit estimate of the
length $a$ of the interval on which the solution does exist. Namely,
$a=min (T, \frac{b}{M})$, where $M:=max_{|u-u_0|\le b,  t\in [0,T]
}|f(t,u)|$, and the neighborhood $D(u_0)$ is taken to be the set
$\{u:\, |u-u_0|\le b\}$.

It is known that {\it in every infinite-dimensional Banach space the
Peano theorem fails}. Therefore, in an infinite-dimensional Banach
space we assume that problems \eqref{e2x} and \eqref{e2y} have a
solution, and if $[0,T)$ is the maximal interval of the existence of
the solution, then Property P holds. This happens, for example, if
$f(t,u)$ satisfies a local Lipschitz condition with respect to $u$
and is continuous with respect to $t\in [0,T]$.  Indeed, if a local
Lipschitz condition holds, then the local interval of the existence
of the solution to the Cauchy problem \eqref{e2y} is of the length
$b=min (RM^{-1}, L)$, provided that $f$ is continuous with respect
to $t$ and satisfies the estimates $||f(t,u)||\le M$,
$||f(t,u)-f(t,v)||\le L||u-v||$, in the region $[0,T]\times
B(u_0,R)$, $B(u_0,R):=\{u: ||u-u_0||\le R\}$. Under these
assumptions the solution to problem \eqref{e2y} is unique and stays
in the ball $B(u_0,R)$ for $t\in [0,b]$.

 To see that Property P holds for problem \eqref{e2y} if $f$ satisfies a
local Lipschitz condition with respect to $u$, assume that the
solution to \eqref{e2y} does not exist for $t>T$. Under our
assumptions, if the solution $u$ of problem \eqref{e2y} satisfies
the inequality $\sup_{0\le t< T} ||u(t)||<\infty$, then the
constants $M, L$ and $R$ are finite. Therefore $b>0$. Take the
initial point $t_0=T-0.5 b$. By the local existence theorem the
solution $u(t)$ exists on the interval $[T-0.5b,T+0.5b]$. This is a
contradiction, since we have assumed that this solution does not
exist for $t>T$. This contradiction proves that Property P holds for
problem \eqref{e2y} if $f$ satisfies a local Lipschitz condition.

 Let us use Theorem 1 to prove asymptotic stability of
the zero solution to \eqref{e2x} and to illustrate the application
of our general method for a study of stability of solutions to
abstract evolution problems, the method that we develop below.

Let $g(t):=||u(t)||$, where the norm is taken in $\R^n$. Take a dot
product of equation \eqref{e2x} with $u$, then take the real part of
both sides of the resulting equation and get
$$\Re(\dot{u},u)=g\dot{g}=\Re (Au,u)+\Re (F(t,u),u)\le -\gamma g^2+cg^{p+1}.$$
Since $g\ge 0$, one obtains from the above inequality  an inequality
of the type \eqref{e7}, namely,
$$\dot{g}(t)\le -\gamma g(t)+cg^p(t), \qquad p=const>1,$$
where $\gamma$ and $c$ are positive constants.
Choose $$\mu(t)=\lambda e^{at},$$ where $\lambda=const>0$, $a=const\in
(0,\gamma)$. Note that $a$ can be chosen arbitrarily close to $\gamma$.
We choose $\lambda$ later. Denote $b:=\gamma -a>0$. Then
inequality
\eqref{e10} holds for any $g(0)$ if $c>0$ is sufficiently small.
Inequality \eqref{e9} holds if $$c\lambda^{-(p-1)}e^{-(p-1)at}\le
\gamma-a=b.$$
Since $p>1$ this inequality holds if $c\lambda^{-(p-1)}<b$.
In turn, the last inequality holds for an arbitrary fixed $c>0$
and an arbitrary small fixed $b>$ provided that $\lambda>0$ is
sufficiently large.

One concludes that {\it for any initial data $u_0$ the solution to
\eqref{e2x} exists globally and admits an estimate $||u(t)||\le
\lambda^{-1}e^{-at}$, where the positive constant $a<\gamma$ can be
chosen arbitrarily close to $\gamma$ if the positive constant $c$ is
sufficiently small.}

{\it The above argument remains valid also for  unbounded, closed, densely
defined linear operators $A(t)$, provided that Property P holds.}

If $A(t)$ is a generator of a $C_0$ semigroup $T(t)$, and $F$
satisfies a local  Lipschitz condition, then problem \eqref{e2x} is
equivalent to the equation $u=T(t)F(t,u)$, and this equation
 may be useful for
a study of the global existence of the solution to problem
\eqref{e2x} (see \cite{P}).

{\bf Example 3.} Consider an example in which {\it the solution
blows up in a finite time}, so it does not exist globally. Consider
the problem \be\label{e3x}\dot{u}-\Delta u=u^2 \quad  in \quad
[0,\infty)\times D\subset \R^n;\quad u_N=0; \quad u(0,x)=u_0(x).\ee
 Here $D$ is a bounded domain with a smooth boundary $S$,
$N$ is an outer unit normal to $S$, $u_0>0$ is a smooth function.
Let
$$g_0:=\int_Du_0(x)dx, \qquad g(t):=\int_Du(t,x)dx.$$ Integrate equation
 \eqref{e3x} over $D$ and get $\dot{g}(t)=\int_Du^2dx$. Use the
inequality $$\big(\int_Dudt\big)^2\le c\int_Du^2dx,$$ where
$c=c(D)=const>0$, and get $\dot{g}\ge g^2/c$. Integrating this
inequality, one obtains $g(t)\ge [\frac 1{g_0}-ct]^{-1}$. Since
$c>0$ and $g_0>0$ it follows that
$$\lim_{t\to t_b}g(t)=\infty,$$
where $t_b:=\frac 1{cg_0}$ is the blow-up time, and $t<t_b$.
Consequently, {\it for any initial data with $g_0>0$ the solution to
\eqref{e3x} does not exist globally.}

{\bf Example 4.} Consider the following equation \be\label{e4x}
\dot{u}+A(t)u +\phi(u)-\psi(u)=f(t,u), \,\, u(0,x)=u_0(x),\ee
 where $u=u(t,x)$, $\phi$ and $\psi (t,u)$ are smooth
functions growing to infinity as $|u|\to \infty$. Let us assume that
$$u\phi(u)\ge 0, \qquad u\psi(t, u)\ge 0 \quad \forall t\ge 0,$$
and
$$u\psi(t,u)\le \alpha(t)|u|^3, \qquad |uf(t,u)|\le \beta(t)|u|, $$
where $\alpha(t)>0$  and $\beta(t)>0$ are continuous functions,
 $x\in D\subset \R^n$,
$D$ is a bounded
domain,
$$\Re (Au,u)\ge \gamma (u,u)\quad \forall u\in D(A), \quad \gamma=const>0,$$
$A$ is an operator in a Hilbert space $H=L^2(D)$, the domain of
definition of $A$, $D(A)$, is a dense in $H$ linear set, $(u,v)$ is
an inner product in $H$, $||u||^2=(u,u)$. An example of $A$ is
$A=-\Delta,$ the Laplacean with the Dirichlet boundary condition on
$S$, the boundary of $D$. Denote $g(t):=||u(t)||$. We want to
estimate the large time behavior of the solution $u$ to \eqref{e4x}.

 Take the inner product in $H$ of \eqref{e4x} and $u$, then take
real part of both sides of the resulting equation and get
$$g\dot{g}\le -\gamma g^2+\alpha g^3 +\beta g.$$ Since $g\ge 0$ one
obtains an inequality of the type \eqref{e7}, namely
$$\dot{g}\le -\gamma g+\alpha(t) g^2 +\beta.$$
Now it is possible to  use Theorem 1.

 Choose $\mu(t)=\lambda
e^{kt}$, where
$\lambda$ and $k$ are positive constants, $k<\gamma$. Assume that
$\lambda g_0\le 1$, where $g_0:=||u_0(x)||$. Then inequality
\eqref{e10} holds for any initial data $u_0$, that is, for any
$g_0$, if $\lambda$ is sufficiently small. Inequality \eqref{e9}
holds if $$\frac{\alpha(t)e^{-kt}}{\lambda}+\lambda e^{kt}\beta(t)\le
\gamma-k.$$ One can easily impose various conditions on $\alpha$ and
$\beta$ so that the above inequality hold. For example, assume that
$\alpha$ decays monotonically as $t$ grows,
$\frac{\alpha(0)}{\lambda}<(\gamma-k)/2$, and $\beta(t)\le \nu e^{-k't}$,
where $k'>k$, $k'=const$, $\nu>0$ is a constant, $\lambda
\nu\le(\gamma-k)/2$. Then inequality \eqref{e9} holds, and it implies
that $$||u(t)||\le \frac{e^{-kt}}{\lambda},$$ so that the exponential decay
of $||u(t)||$ as $t\to \infty$ is established.

In Sections 3 and 4 some stability results for abstract evolution problems
are presented in detail. These results are formulated in four theorems.
The basic ideas are similar to the ones discussed in examples in this
Section, but new assumptions and new technical tools are used.

\section{Stability results }

In this Section we develop a method for a study of stability
of solutions to the evolution problems described by the Cauchy
problem \eqref{e1}-\eqref{e2} for abstract differential equations
with a dissipative bounded  linear operator $A(t)$ and a nonlinearity
$F(t,u)$ satisfying inequality \eqref{e4}. Condition \eqref{e4} means that
for
sufficiently small $\|u(t)\|$ the nonlinearity is of the higher
order of smallness than $\|u(t)\|$. We also study the large time behavior
of the solution to problem \eqref{e1}-\eqref{e2} with
persistently acting perturbations $b(t)$.

In this paper we assume that
$A(t)$ is a bounded linear dissipative operator, but our methods are
valid also for unbounded linear dissipative operators $A(t)$,
for which one can prove global existence of the solution to
problem \eqref{e1}-\eqref{e2}. We do not go into further detail
in this paper.

Let us formulate the first stability result.

{\bf Theorem 3.} {\it Assume that Re$(Au,u)\le -k\|u\|^2$
$\forall u\in H$, $k=const>0$, and inequality \eqref{e3} holds with
$\gamma(t)=k$.
Then the solution
to problem \eqref{e1}-\eqref{e2} with $b(t)=0$ satisfies an esimate
$\|u(t)\|=O(e^{-(k-\epsilon)t})$ as $t\to \infty$. Here
$0<\epsilon< k$ can be chosen arbitrarily small if
$\|u_0\|$ is sufficiently small.}

This theorem implies asymptotic stability in the sense of Lyapunov of
the zero solution to equation \eqref{e1} with $b(t)=0$. Our proof of
Theorem 3 is new and very short.

{\it Proof of Theorem 3}.
Multiply equation \eqref{e1} (in which $b(t)=0$ is assumed) by $u$, denote
$g=g(t):=\|u(t)\|$, take the real part, and use  assumption
\eqref{e3} with $\gamma(t)=k>0$, to get
\begin{equation}
\label{e25}
g\dot{g}\le -kg^2+c_0g^{p+1}, \qquad p>1.
\end{equation}
If $g(t)>0$ then the derivative $\dot{g}$ does exist, and
$$\dot{g}(t)=Re  \left(\dot{u}(t), \frac {u(t)}{\|u(t)\|}\right),$$
as one can check.
If $g(t)=0$ on an open subset of $\R_+$, then
the derivative $\dot{g}$ does exist on this subset and $\dot{g}(t)=0$
on this subset. If $g(t)=0$ but in in any neighborhood $(t-\delta,
t+\delta)$ there are points at which $g$ does not vanish,
then by $\dot{g}$ we understand the derivative from the right,
that is,
$$\dot{g}(t):= \lim_{s\to +0}\frac {g(t+s)-g(t)}{s}=\lim_{s\to +0}\frac
{g(t+s)}{s}.$$
This limit does exist and is equal to $\|\dot{u}(t)\|$.
Indeed, the function $u(t)$ is continuously differentiable,
so
$$\lim_{s\to +0}\frac {\|u(t+s)\|}{s}=\lim_{s\to +0}
\frac{\|s\dot{u}(t)+o(s)\|}{s}=\|\dot{u}(t)\|.$$
The assumption about the existence of the bounded derivative $\dot{g}(t)$
from the right in Theorem 3 was made because the function $\|u(t)\|$
does not have, in general, the derivative in the usual sense at the points
$t$ at which $\|u(t)\|=0$, no matter how smooth the function $u(t)$
is at the point $\tau$. Indeed,
$$\lim_{s\to -0}\frac {\|u(t+s)\|}{s}=\lim_{s\to -0}
\frac{\|s\dot{u}(t)+o(s)\|}{s}=-\|\dot{u}(t)\|,$$
because $\lim_{s\to -0}\frac {|s|}{s}=-1$. Consequently,
the right and left derivatives of $\|u(t)\|$ at the point $t$
at which $\|u(t)\|=0$ do exist, but are different. Therefore,
the derivative of $\|u(t)\|$ at the point $t$
at which $\|u(t)\|=0$ does not exist in the usual sense.

However, as we have proved above, the derivative
$\dot{g}(t)$ from the right does exist always, provided that $u(t)$ is
continuously
differentiable at the point $t$.

Since $g\ge 0$, inequality \eqref{e25}  yields inequality \eqref{e7}
with
$\gamma(t)=k>0$, $\beta(t)=0$, and $\alpha(t,g)=c_0 g^p$, $p>1$.
Inequality \eqref{e9} takes the form
\begin{equation}
\label{e26}
\frac {c_0}{\mu^p(t)}\leq \frac 1 {\mu(t)}\left(k-\frac
{\dot{\mu}(t)}{\mu(t)} \right), \qquad  \forall t\ge 0.
\end{equation}
Let
\begin{equation}
\label{e27}
\mu(t)=\lambda e^{bt}, \qquad \lambda,b=const>0.
\end{equation}
We choose the constants $\lambda$ and $b$ later.
Inequality \eqref{e9}, with $\mu$ defined in \eqref{e27}, takes the form
\begin{equation}
\label{e28}
\frac {c_0}{\lambda^{p-1}e^{(p-1)bt}} +b\leq k, \qquad \forall
t\ge 0.
\end{equation}
This inequality holds if it holds at $t=0$, that is, if
\begin{equation}
\label{e29}
\frac {c_0}{\lambda^{p-1}} +b\leq k.
\end{equation}
Let $\epsilon>0$ be arbitrary small number. Choose
$b=k-\epsilon>0$. Then  \eqref{e29} holds if
\begin{equation}
\label{e30}
\lambda\geq \big(\frac{c_0}{\epsilon} \big)^{\frac 1 {p-1}}.
\end{equation}
Condition \eqref{e10} holds if
\begin{equation}
\label{e31}
\|u_0\|=g(0)\le \frac 1 {\lambda}.
\end{equation}
We choose $\lambda$ and $b$ so that inequalities \eqref{e30}
and \eqref{e31} hold. This is always possible if $b<k$ and $\|u_0\|$
is sufficiently small.

By Theorem 1, if  inequalities  \eqref{e29}-\eqref{e31} hold, then
one gets estimate  \eqref{e11}:
\begin{equation}
\label{e32}
0\le g(t)=\|u(t)\|\le \frac {e^{-(k-\epsilon)t}} {\lambda},\qquad
\forall t\ge 0.
\end{equation}
Theorem 3 is proved. \hfill $\Box$

{\bf Remark 3.} {\it One can formulate the result differently.
Namely, choose  $\lambda=\|u_0\|^{-1}$. Then inequality \eqref{e31}
holds, and becomes an equality.
Substitute this $\lambda$ into \eqref{e29} and get
$$c_0\|u_0\|^{p-1}+b\leq k.$$
Since the choice of the constant $b>0$ is
at our disposal, this inequality can always be satisfied if
$c_0\|u_0\|^{p-1}<k$.
Therefore, condition
$$c_0\|u_0\|^{p-1}<k$$
is a sufficient condition for
the estimate
$$\|u(t)\|\le \|u_0\|e^{-(k-c_0\|u_0\|^{p-1})t},$$
to hold (assuming  that $c_0\|u_0\|^{p-1}<k$).
}

Let us formulate the second stability result.

{\bf Theorem 4.} {\it Assume that inequalities \eqref{e3}-\eqref{e5}
hold and
\begin{equation}
\label{e33}
\gamma(t)=\frac {c_1}{(1+t)^{q_1}}, \quad q_1\le 1; \quad c_1,q_1=const>0.
\end{equation}
Suppose that $\epsilon\in (0,c_1)$ is an arbitrary small fixed number,
$$\lambda\ge \left(\frac{c_0}{\epsilon}\right)^{1/(p-1)} \quad \text{
and}\quad \|u(0)\|\le \frac {1}{\lambda}.$$
Then the unique solution to \eqref{e1}-\eqref{e2} with $b(t)=0$
exists on all of $\R_+$ and
\begin{equation}
\label{e34}
0\le \|u(t)\|\le \frac {1}{\lambda(1+t)^{c_1-\epsilon}},\qquad \forall
t\ge 0.
\end{equation}
}

Theorem 4 gives the size of the initial data, namely, $\|u(0)\|\le \frac
{1}{\lambda}$, for which estimate \eqref{e34} holds. For a fixed
nonlinearity $F(t,u)$, that is, for a fixed constant $c_0$ from
assumption \eqref{e4}, the maximal size of $\|u(0)\|$ is determined by the
minimal size of $\lambda$.

The minimal size of $\lambda$ is determined by the inequality
$\lambda\ge \left(\frac{c_0}{\epsilon}\right)^{1/(p-1)}$, that is, by
the maximal size of $\epsilon\in (0,c_1)$.  If $\epsilon<c_1$ and
$c_1-\epsilon$ is very small, then
$\lambda>\lambda_{min}:= \left(\frac{c_0}{c_1}\right)^{1/(p-1)}$
and $\lambda$ can be chosen very close to $\lambda_{min}$.

{\it Proof of Theorem 4.}
Let
\begin{equation}
\label{e35}
\mu(t)=\lambda (1+t)^\nu, \qquad \lambda, \nu=const>0.
\end{equation}
We will choose the constants $\lambda$ and $\nu$ later.
Inequality \eqref{e9} (with $\beta(t)=0$) holds if
\begin{equation}
\label{e36}
\frac{c_0}{\lambda^{p-1} (1+t)^{(p-1)\nu}} +\frac {\nu}{1+t}\le \frac
{c_1}{(1+t)^{q_1}}, \qquad \forall t\ge 0.
\end{equation}
If
\begin{equation}
\label{e37}
q_1\le 1\qquad and  \qquad (p-1)\nu\ge q_1,
\end{equation}
then inequality \eqref{e36} holds if
\begin{equation}
\label{e38}
\frac {c_0}{\lambda^{p-1}} +\nu\le c_1.
\end{equation}
Let $\epsilon>0$ be an arbitrary small number. Choose
\begin{equation}
\label{e39}
\nu= c_1-\epsilon.
\end{equation}
Then inequality \eqref{e38}  holds if inequality \eqref{e30} holds.
Inequality \eqref{e10}
holds because we have assumed in Theorem 4 that $\|u(0)\|\le \frac 1
\lambda$.
Combining inequalities  \eqref{e30}, \eqref{e31} and \eqref{e11}, one
obtains the desired estimate:
\begin{equation}
\label{e40}
0\le \|u(t)\|=g(t)\le \frac 1 {\lambda (1+t)^{c_1 -\epsilon}}, \qquad
\forall t\ge 0.
\end{equation}
Condition  \eqref{e30} holds for any fixed small $\epsilon>0$ if
$\lambda$ is sufficiently large. Condition  \eqref{e31} holds for any
fixed large $\lambda$ if $\|u_0\|$ is sufficiently small.

Theorem 4 is proved. \hfill $\Box$

Let us formulate a stability result in which we assume that
$b(t)\not\equiv 0$.
The function $b(t)$ has physical meaning of persistently acting
perturbations.

{\bf Theorem 5.} {\it Let $b(t)\not\equiv 0$, conditions \eqref{e3}-
\eqref{e5} and \eqref{e33} hold, and
\begin{equation}
\label{e41}
\beta(t)\le \frac {c_2}{(1+t)^{q_2}},
\end{equation}
where $c_2>0$ and $q_2>0$ are constants.
Assume that
\begin{equation}
\label{e42}
q_1\le \min\{1, q_2-\nu, \nu(p-1)\},\qquad \|u(0)\|\le \lambda_0^{-1},
\end{equation}
where $\lambda_0>0$ is a constant defined in \eqref{e49}, and
\begin{equation}
\label{e43}
c_2^{1-\frac 1 p}c_0^{\frac 1 p}(p-1)^{\frac 1 p}\frac p{p-1} +\nu\le c_1.
\end{equation}
Then problem \eqref{e1}-\eqref{e2} has a unique global solution $u(t)$,
and the following  estimate holds:
\begin{equation}
\label{e44}
\|u(t)\|\le \frac 1 {\lambda_0 (1+t)^\nu}, \qquad \forall t\ge 0.
\end{equation}
}

{\it Proof of Theorem 5.} Let $g(t):=\|u(t)\|$. As in the proof
of Theorem 4, multiply \eqref{e1} by $u$, take the real part, use the
assumptions of Theorem 5, and get the inequality:
\begin{equation}
\label{e45}
\dot{g}\le -\frac {c_1}{(1+t)^{q_1}}g +c_0g^p+\frac {c_2}{(1+t)^{q_2}}.
\end{equation}
Choose $\mu(t)$ by formula \eqref{e35}. Apply Theorem 1 to
inequality \eqref{e45}. Condition \eqref{e9} takes now the form
\begin{equation}
\label{e46}
\frac {c_0}{\lambda^{p-1} (1+t)^{(p-1)\nu}}+ \frac
{\lambda c_2}{(1+t)^{q_2-\nu}}+ \frac \nu{1+t}\le \frac
{c_1}{(1+t)^{q_1}}\quad \forall t\ge 0.
\end{equation}
If assumption \eqref{e42} holds, then inequality \eqref{e46} holds
provided that it holds for $t=0$, that is, provided that
\begin{equation}
\label{e47}
\frac {c_0}{\lambda^{p-1}}+ \lambda c_2+ \nu \le c_1.
\end{equation}
Condition  \eqref{e10} holds if
\begin{equation}
\label{e48}
g(0)\leq \frac 1 \lambda.
\end{equation}
The function  $h(\lambda):=\frac {c_0}{\lambda^{p-1}}+ \lambda
c_2$ attains its global minimum in the interval $[0,\infty)$
at the value
\begin{equation}
\label{e49}
\lambda=\lambda_0:=\left(\frac {(p-1)c_0}{c_2}\right)^{1/p},
\ee
and this minimum
is equal to
$$h_{min}=c_0^{\frac 1 p}c_2^{1-\frac 1 p}(p-1)^{\frac 1
p}\frac p{p-1}.$$
Thus, substituting $\lambda=\lambda_0$ in formula \eqref{e47},
one concludes that inequality \eqref{e47} holds if
the following inequality holds:
\begin{equation}
\label{e50}
c_0^{\frac 1 p}c_2^{1-\frac 1 p}(p-1)^{\frac 1 p}\frac p{p-1}+\nu\le c_1,
\end{equation}
while inequality \eqref{e48} holds if
\begin{equation}
\label{e51}
\|u(0)\|\le \frac 1 {\lambda_0}.
\end{equation}
Therefore, by Theorem 1, if conditions \eqref{e50}-\eqref{e51} hold, then
estimate \eqref{e11}
yields
\begin{equation}
\label{e52}
\|u(t)\|\le \frac 1 {\lambda_0 (1+t)^\nu}, \qquad \forall t\ge 0,
\end{equation}
where $\lambda_0$ is defined in \eqref{e49}.

Theorem 5 is proved. \hfill $\Box$

\section{Stability results under non-classical assumptions }

Let us assume that Re$(A(t)u,u)\le \gamma(t)\|u\|^2$, where $\gamma(t)>0$.
This corresponds to the case when the linear operator $A(t)$ may have
spectrum in the right half-plane Re$z>0$. Our goal is to derive under this
assumption sufficient conditions on $\gamma(t)$, $\alpha(t,g)$, and
$\beta(t)$, under which the solution to problem \eqref{e1} is bounded as
$t\to \infty$, and stable. We want to demonstrate new methodology, based
on Theorem 1. By this reason we restrict ourselves to a derivation of the
simplest results under simplifying assumptions. However, our derivation
illustrates the method applicable in many other problems.

Our assumptions in this Section are:
$$\beta(t)=0,\quad
\gamma(t)=c_1(1+t)^{-m_1}, \quad \alpha(t,g)=c_2(1+t)^{-m_2}g^p,
\,\,p>1.$$
Let us choose
$$\mu(t)=d+\lambda(1+t)^{-n}.$$
The constants $c_j, m_j, \lambda,
d, n,$ are assumed positive.

We want to show that a suitable choice of these parameters allows one to
check that basic inequality \eqref{e9} for $\mu$ is satisfied,
and, therefore, to obtain inequality \eqref{e11} for $g(t)$.
This inequality allows one to derive global boundedness of the solution to
\eqref{e1}, and the Lyapunov stability of the zero solution to
\eqref{e1} (with $u_0=0$). Note that under our assumptions
$\dot{\mu}<0$, $\lim_{t \to \infty}\mu(t)=d$. We choose $\lambda=d$.
Then $(2d)^{-1}\le \mu^{-1}(t)\le d^{-1}$ for all $t\ge 0$.
The basic inequality  \eqref{e9} takes the form
\begin{equation}
\label{e53}
c_1(1+t)^{-m_1}+c_2(1+t)^{-m_2}[d+\lambda(1+t)^{-n}]^{-p+1}\le
n\lambda (1+t)^{-n-1}[d+\lambda(1+t)^{-n}]^{-1},
\end{equation}
and
\begin{equation}
\label{e54}
g_0(d+\lambda)\le 1.
\end{equation}
Since we have chosen $\lambda=d$, condition \eqref{e54}
is satisfied if
\begin{equation}
\label{e55}
d=(2g_0)^{-1}.
\end{equation}
Choose $n$ so that
\begin{equation}
\label{e56}
n+1\leq \min\{m_1, m_2\}.
\end{equation}
Then \eqref{e53} holds if
\begin{equation}
\label{e57}
c_1+c_2d^{-p+1}\le n\lambda d^{-1}.
\end{equation}
Inequality \eqref{e57} is satisfied if $c_1$ and $c_2$ are sufficiently
small.
Let us formulate our result, which folows from Theorem 1.

{\bf Theorem 6}. {\it If inequalities  \eqref{e57} and  \eqref{e56} hold,
then
\begin{equation}
\label{e58}
0\le g(t)\le [d+\lambda(1+t)^{-n}]^{-1}\le d^{-1}, \qquad \forall t\ge 0.
\end{equation}
}
Estimate \eqref{e58} proves global boundedness of the solution $u(t)$,
and implies Lyapunov stability of the zero solution to problem
 \eqref{e1} with $b(t)=0$ and $u_0=0$.

Indeed, by the definition
of Lyapunov stability of the zero solution, one should check that
for an arbitrary small fixed $\epsilon>0$ estimate $\sup_{t\ge
0}\|u(t)\|\le \epsilon$  holds provided that $\|u(0)\|$ is sufficiently
small. Let $\|u(0)\|=g_0=\delta$. Then estimate \eqref{e58}
yields $\sup_{t\ge 0}\|u(t)\|\le d^{-1}$, and \eqref{e55}
implies  $\sup_{t\ge 0}\|u(t)\|\le 2\delta$. So,
$\epsilon=2\delta$, and the Lyapunov stability is proved. \hfill $\Box$


\end{document}